\documentclass[11pt]{amsart}

\usepackage{amsmath,amssymb,amsthm}
\usepackage{hyperref}
\hypersetup{hidelinks}
\usepackage{tikz}

\newtheorem{theorem}{Theorem}[section]
\newtheorem{lemma}[theorem]{Lemma}
\newtheorem{proposition}[theorem]{Proposition}
\newtheorem{corollary}[theorem]{Corollary}
\theoremstyle{definition}
\newtheorem{definition}[theorem]{Definition}
\newtheorem{remark}[theorem]{Remark}

\newtheorem{question}[theorem]{Question}

\title[Crossing Numbers of Knots on Closed Surfaces]
{Crossing Numbers of Knots on Closed Surfaces}
\author{Makoto Ozawa}
\address{Department of Natural Sciences, Faculty of Arts and Sciences, Komazawa University, 1-23-1 Komazawa, Setagaya-ku, Tokyo, 154-8525, Japan}
\email{w3c@komazawa-u.ac.jp}
\date{\today}
\dedicatory{Dedicated to Professor Kanji Morimoto on the occasion of his retirement.}
\begin{document}
\begin{abstract}
For a knot $K$ and a closed connected surface $F\subset S^{3}$, we define the
surface crossing number $c(K;F)$ as the minimum number of crossings among
regular projections of representatives of $K$ to $F$.  Let
$S^{3}=M_1\cup_F M_2$ and let
\[
\delta(F)=g(M_1;F)+g(M_2;F)-g(F),
\]
where $g(M_i;F)$ denotes relative Heegaard genus.  We prove that the
zero-crossing case is controlled by tunnel number: if $c(K;F)=0$, then
$t(K)\le \delta(F)$.  Consequently, if $t(K)>\delta(F)$, then $c(K;F)>0$ and
\[
c(K;F)\ge 2\bigl(t(K)-\delta(F)\bigr)+1.
\]
Thus $\delta(F)$ is the amount of tunnel complexity that can be absorbed by the
surface without creating crossings.  The proof combines surface ascending
number, surface bridge number, and Heegaard amalgamation.  For each fixed $F$,
$c(K;F)$ is unbounded, and explicit connected-sum families exhibit linear
growth.
\end{abstract}
\maketitle

\section{Introduction}
\subsection{Background and Motivation}
The classical crossing number is defined from planar diagrams, and hence from projections to $S^2$.

In this paper we consider diagrams on a closed surface $F \subset S^3$. For a knot $K$, the \emph{surface crossing number} $c(K;F)$ is the minimal number of crossings among all regular diagrams obtained by isotoping $K$ into a regular neighborhood of $F$.

Our main result identifies $\delta(F)$ as a threshold: if the tunnel number of $K$ exceeds this threshold, then crossings on $F$ are forced and satisfy a linear lower bound.

\subsection{Main Results}
The key point is that the case $c(K;F)=0$ should not be treated as a small
crossing-number case.  It is a different geometric regime.  If a knot has a
crossing-free diagram on $F$, then the complexity of the knot is carried by the
surface and by the two complementary 3-manifolds cut off by $F$.  The correct
lower bound therefore begins only after the tunnel number of the knot exceeds
this ambient threshold.

Let $S^3=M_1\cup_F M_2$ be the decomposition determined by the closed surface
$F$, and put
\[
\delta(F)=g(M_1;F)+g(M_2;F)-g(F).
\]
The precise statements are given in Section~\ref{sec:main_results}.  In brief,
the zero-crossing stratum satisfies
\[
c(K;F)=0 \quad\Longrightarrow\quad t(K)\le \delta(F),
\]
and the complementary threshold estimate says that if
\[
t(K)>\delta(F),
\]
then $c(K;F)>0$ and
\[
c(K;F)\ge 2\bigl(t(K)-\delta(F)\bigr)+1.
\]
Thus $\delta(F)$ is exactly the amount of tunnel complexity that can be hidden
by placing $K$ on the surface without crossings.

The proof explains the origin of the threshold.  In the positive-crossing case
one has the diagrammatic estimate
\[
\frac{c(K;F)-1}{2}\ge a(K;F),
\]
where $a(K;F)$ is the surface ascending number.  The subtraction of $1$ comes
from choosing one actual crossing as a base crossing.  When $c(K;F)=0$, no such
base crossing exists, and the above estimate is simply not the right statement.
Instead, the remaining inequalities
\[
a(K;F)\ge b(K;F)-1\ge t(K)-\delta(F)
\]
show that a crossing-free diagram forces $t(K)\le\delta(F)$.

For connected sums
\[
K_m=\underbrace{K_0\#\cdots\#K_0}_{m\text{ times}},
\]
where $t(K_0)=1$ and $t(K_m)\ge m$, the main theorem gives, for all
$m>\delta(F)$,
\[
c(K_m;F)\ge 2m-2\delta(F)+1.
\]
Thus the surface crossing number grows linearly in this explicit family.

\subsection{Organization of the Paper}
Section 2 recalls the necessary definitions and proves monotonicity under compression.
Section 3 states the main results.
Section 4 proves the zero-crossing obstruction and the threshold inequality.
Section 5 discusses a family of connected sums giving linear growth.
Section 6 records some related questions and remarks.

\section{Preliminaries}
In this section we recall basic notions concerning knots on closed surfaces, tunnel number, and Heegaard splittings. Throughout the paper all manifolds are assumed to be compact, connected, and orientable.

\subsection{Handlebodies, compression bodies, and Heegaard genus}
We recall the standard definitions of compression bodies and Heegaard splittings, following Casson and Gordon \cite{CG}.

\begin{definition}
A \textit{compression body} $W$ is a compact, connected, orientable 3-manifold constructed by starting with a product $S \times [0, 1]$, where $S$ is a closed, orientable (possibly disconnected) surface, and attaching finitely many 1-handles to $S \times \{1\}$.
We denote $\partial_{-} W = S \times \{0\}$ and $\partial_{+} W = \partial W \setminus \partial_{-} W$.
If $\partial_{-} W = \emptyset$, that is, $W$ is constructed by attaching 1-handles to a 3-ball (viewed as a 0-handle), then $W$ is called a \textit{handlebody}.
\end{definition}

\begin{definition}
Let $M$ be a compact, connected, orientable 3-manifold. A \textit{Heegaard splitting} of $M$ is a decomposition $M = V \cup_{H} W$, where $V$ and $W$ are compression bodies such that $V \cap W = \partial_{+} V = \partial_{+} W = H$. The surface $H$ is called the \textit{Heegaard surface} of the splitting.
\end{definition}

When $M$ is a closed 3-manifold, both $V$ and $W$ are necessarily handlebodies. The \textit{Heegaard genus} of a closed 3-manifold $M$, denoted by $g(M)$, is the minimal genus of a Heegaard surface among all Heegaard splittings of $M$.

For the complementary pieces cut off by a closed surface, we shall use the following boundary-relative version of Heegaard genus.

\begin{definition}[Relative Heegaard genus]\label{def:relative_heegaard_genus}
Let $M$ be a compact, connected, orientable 3-manifold with connected boundary $F$.  A \emph{Heegaard splitting of $M$ relative to $F$} is a decomposition
\[
M=C\cup_{\Sigma}H,
\]
where $C$ is a compression body with $\partial_-C=F$, $H$ is a handlebody, and
\[
C\cap H=\partial_+C=\partial H=\Sigma.
\]
The \emph{Heegaard genus of $M$ relative to $F$}, denoted $g(M;F)$, is the minimal genus of such a surface $\Sigma$.  When the boundary surface $F$ is fixed and no confusion is possible, we write simply $g(M)$ for $g(M;F)$.
\end{definition}

\subsection{Knots and diagrams on closed surfaces}
Let $F \subset S^{3}$ be a closed embedded surface. To define a diagram of a knot $K$ on $F$, we consider a closed regular neighborhood $N(F) \cong F \times [-1, 1]$ of $F$ in $S^{3}$, identifying $F$ with $F \times \{0\}$.

By an ambient isotopy of $S^{3}$, we may assume that $K$ is entirely contained in the interior of $N(F)$. The canonical projection $p \colon F \times [-1, 1] \to F$ maps $K$ onto the surface $F$. We say that $K$ is in \textit{regular position} with respect to $F$ if the restriction $p|_{K}$ is an immersion whose only singularities are finitely many transverse double points. Such a projection, equipped with the standard over/under crossing information derived from the $[-1, 1]$ coordinate, yields a \textit{diagram} of $K$ on $F$.

\begin{definition}[Surface crossing number]
The \textit{surface crossing number} $c(K;F)$ is defined as the minimal number of crossings among all regular diagrams of $K$ on $F$. This minimum is taken over all ambient isotopies of $K$ into $N(F)$ that yield regular projections. When $F=S^{2}$, this invariant coincides with the classical crossing number $c(K)$.
\end{definition}

\begin{remark}\label{rem:finite_crossing}
The invariant $c(K;F)$ is finite.  Indeed, the product neighborhood $N(F)$ contains a 3-ball.  Every knot type has a representative contained in this ball, and every classical planar diagram of $K$ may be realized in a disk contained in $F$.  Hence
\[
c(K;F)\le c(K).
\]
\end{remark}

\begin{remark}
Even for a fixed surface $F$, the value of $c(K;F)$ depends on the isotopy class of $K$ inside $N(F)$. Thus one must minimize over all regular positions of $K$ with respect to $F$.
\end{remark}

\subsection{The surface bridge number}
The surface bridge number $b(K;F)$ extends the classical bridge number to a closed surface $F$.  We use a definition that is compatible with bridge decompositions with respect to Heegaard surfaces.  Thus a bridge arc is required not only to have the correct height behavior, but also to be boundary-parallel in the corresponding product region.

\begin{definition}[Diagrammatic bridge presentation]
Let $D$ be a regular diagram of a knot $K$ on $F$.  An \emph{over-bridge} of $D$ is a maximal subarc of the diagram which contains only over-crossings, and an \emph{under-bridge} is defined similarly.  We say that $D$ admits a \emph{$b$-bridge presentation} if it is the union of $b$ over-bridges and $b$ under-bridges.
\end{definition}

\begin{definition}[Geometric bridge presentation]\label{def:surface_bridge}
Let $N(F)\cong F\times[-1,1]$ be a regular neighborhood of $F$, with $F=F\times\{0\}$, and let
\[
h:N(F)\to[-1,1]
\]
be the projection to the second factor.  A knot $K\subset N(F)$ is in \emph{$b$-bridge position with respect to $F$} if the following conditions hold:
\begin{enumerate}
\item $h|_K$ is Morse;
\item all maxima of $h|_K$ lie in $F\times(0,1]$ and all minima lie in $F\times[-1,0)$;
\item $K\cap (F\times[0,1])$ is a union of $b$ arcs, each boundary-parallel in the product $F\times[0,1]$;
\item $K\cap (F\times[-1,0])$ is a union of $b$ arcs, each boundary-parallel in the product $F\times[-1,0]$.
\end{enumerate}
Equivalently, each upper bridge arc admits a bridge disk to $F\times\{0\}$ in $F\times[0,1]$, and each lower bridge arc admits a bridge disk to $F\times\{0\}$ in $F\times[-1,0]$.
The \emph{surface bridge number} $b(K;F)$ is the minimal such integer $b$, taken over all isotopic positions of $K$ in $N(F)$.
\end{definition}

\begin{remark}\label{rem:diagram_geometric_bridge}
The diagrammatic and geometric formulations agree.  Given a diagram with $b$
over-bridges and $b$ under-bridges, realize the over-bridges slightly above
$F$ and the under-bridges slightly below $F$.  After a small perturbation, each
bridge arc is the graph of an arc in $F$ and is therefore boundary-parallel in
the corresponding product region; the product trace gives a bridge disk.
Conversely, projecting a geometric $b$-bridge position to $F$ gives a diagram
with $b$ over-bridges and $b$ under-bridges after a small perturbation.  This
is the surface analogue of the usual passage between bridge diagrams and
bridge positions; compare Doll's generalized bridge positions for Heegaard
splittings \cite{Doll}.
\end{remark}

\begin{remark}\label{rem:bridge_positive}
For every knot $K$ and every closed surface $F$, one has
\[
b(K;F)\ge 1.
\]
Indeed, for any Morse function on a closed 1-manifold, at least one maximum and at least one minimum occur.
\end{remark}

This geometric definition is a surface version of the generalized bridge number introduced by Doll \cite{Doll}, who studied knots and links with respect to Heegaard splittings of 3-manifolds.

\subsection{The surface ascending number}
The ascending number of a knot was introduced by the author in \cite{O} as a measure of the deviation of a knot diagram from a trivial descending diagram. We extend this concept to a closed surface $F$ to define the surface ascending number $a(K;F)$.

\begin{definition}
Let $D$ be an oriented regular diagram of a knot $K$ on a closed surface $F$, and let $x$ be a base point on $D$ distinct from any crossing. As we traverse $D$ starting from $x$ in the direction of the orientation, we pass through each crossing exactly twice. A crossing of $D$ is called a \emph{ascending crossing} if we first encounter it as an under-crossing.

Let $a(D, x)$ denote the number of ascending crossings in $D$ with respect to the base point $x$. The \emph{surface ascending number} $a(K;F)$ is defined as the minimal value of $a(D, x)$, where the minimum is taken over all oriented regular diagrams $D$ of $K$ on $F$ and all choices of the base point $x$.
\end{definition}


\subsection{Monotonicity under compression}
The surface crossing number is monotone under compression of the ambient surface.

\begin{proposition}\label{prop:compression}
Let $F$ and $F'$ be closed surfaces in $S^{3}$. If $F'$ is obtained from $F$ by a compression, then
\[
c(K;F) \le c(K;F').
\]
\end{proposition}

\begin{proof}
Dually, $F$ is obtained from $F'$ by attaching a 1-handle, or tube, along the capping disks created by the compression.  This dual 1-handle attachment recovers the original isotopy class of $F$.
Let $D$ be a regular diagram of $K$ on $F'$ realizing $c(K;F')$, and regard the underlying diagram as a finite 1-complex on $F'$.

The attaching disks of the tube may be chosen in the capping disks created by the compression.  After making the diagram transverse to the boundaries of these disks, an outermost-arc and innermost-circle argument removes intersections of the diagram with the disk interiors: since each attaching disk is simply connected, every arc of the diagram lying in such a disk can be pushed to a small collar of the disk boundary, and closed components are removed similarly.  This isotopy is performed on the surface $F'$ and does not create new crossings.

If the compression is separating, $F'$ may be disconnected.  In that case the diagram of the knot lies on one component of $F'$, and the inverse 1-handle is attached to the appropriate capping disks; the same disk-disjointness argument applies componentwise.  Thus the attaching disks can be taken disjoint from the diagram of $K$.  Attaching the tube then realizes the same diagram on the higher-genus surface $F$, with the same number of crossings.  Hence $c(K;F)\le c(K;F')$.
\end{proof}

Repeated compression to 2-spheres gives
\[
c(K;F) \le c(K;S^{2}) = c(K).
\]

\begin{remark}
The same monotonic behavior holds for the surface ascending number and the surface bridge number. By a similar isotopy argument, the 1-handle can be attached without altering the ascending points of a diagram or the configuration of a bridge presentation. This yields $a(K;F) \le a(K;F')$ and $b(K;F) \le b(K;F')$, hence $a(K;F) \le a(K)$ and $b(K;F) \le b(K)$ for any closed surface $F$.
\end{remark}

\subsection{Tunnel number}
Let $E(K)=S^{3}\setminus \operatorname{int} N(K)$ denote the exterior of $K$. The tunnel number was introduced by Clark \cite{C} and is related to Heegaard splittings of the knot exterior \cite{CG}.

\begin{definition}[Clark \cite{C}]
The \textit{tunnel number} of $K$, denoted $t(K)$, is the minimal number of properly embedded arcs $\{\tau_{1},\dots,\tau_{t}\}\subset E(K)$ such that the complement $E(K)\setminus \operatorname{int}N(\tau_{1}\cup\cdots\cup\tau_{t})$ is a handlebody.
\end{definition}

\begin{remark}
Adding $t(K)$ tunnels to the knot neighborhood yields a Heegaard splitting of $S^3$, and the boundary of the resulting handlebody provides a Heegaard surface for $E(K)$. Consequently, $t(K) + 1 = g(E(K))$.
\end{remark}

\subsection{Heegaard splittings determined by a closed surface}
Let $F\subset S^{3}$ be a closed connected surface.  Since $H_{2}(S^{3};\mathbb Z)=0$, every closed orientable surface in $S^{3}$ is separating.  We nevertheless sometimes keep the word ``separating'' in the statements below to emphasize the decomposition
\[
S^{3}=M_{1}\cup_{F}M_{2}.
\]
The manifolds $M_{1}$ and $M_{2}$ need not be handlebodies. The following quantity measures how far $F$ is from being a Heegaard surface.

\begin{definition}[Heegaard deficiency]\label{def:deficiency}
The Heegaard deficiency of $F$ is
\[
\delta(F) := g(M_{1};F)+g(M_{2};F)-g(F).
\]
Equivalently, using the convention of Definition~\ref{def:relative_heegaard_genus}, we write
\[
\delta(F)=g(M_{1})+g(M_{2})-g(F),
\]
where each $g(M_i)$ denotes the Heegaard genus of $M_i$ relative to the boundary component $F$.
\end{definition}

\begin{remark}\label{rem:deficiency}
We have in fact $\delta(F) \ge g(F)$.  Indeed, let
\[
M_i=C_i\cup_{\Sigma_i}H_i
\]
be a Heegaard splitting of $M_i$ relative to $F$, so that $\partial_-C_i=F$.  The compression body $C_i$ is obtained from $F\times[0,1]$ by attaching 1-handles to $F\times\{1\}$.  Attaching 1-handles cannot decrease the genus of the positive boundary.  Hence
\[
g(\Sigma_i)=g(\partial_+C_i)\ge g(\partial_-C_i)=g(F).
\]
Taking the minimum over all relative Heegaard splittings gives $g(M_i;F)\ge g(F)$ for $i=1,2$, and therefore
\[
\delta(F)=g(M_1;F)+g(M_2;F)-g(F)\ge g(F).
\]

When $F$ is a Heegaard surface of $S^{3}$, the manifolds $M_{1}$ and $M_{2}$ are handlebodies of genus $g(F)$, so $g(M_i;F)=g(F)$ for $i=1,2$ and $\delta(F)=g(F)$.  In general, $\delta(F)$ measures how far $F$ is from being a Heegaard surface.
\end{remark}

\subsection{\texorpdfstring{The $h$-genus}{The h-genus}}
We recall the invariant measuring how efficiently a knot can lie on an unknotted surface.

\begin{definition}[$h$-genus, \cite{M2}]
Let $H\subset S^{3}$ be an unknotted closed surface (i.e.\ a Heegaard surface). Define
\[
h(K) = \min \{ g(H) \mid c(K;H)=0 \}.
\]
Thus $h(K)$ is the minimal genus of a Heegaard surface on which $K$ admits a crossing-free diagram.
\end{definition}

\begin{proposition}[{\cite[Fact]{M2}}]
For any knot $K\subset S^{3}$,
\[
t(K)\le h(K).
\]
\end{proposition}

\begin{remark}\label{rem:h_genus_from_zero}
This inequality is also recovered as a special case of
Proposition~\ref{prop:zero}.  Indeed, if a Heegaard surface $H$ realizes
$h(K)$, then $c(K;H)=0$ and $\delta(H)=g(H)$.  Hence
Proposition~\ref{prop:zero} gives $t(K)\le g(H)=h(K)$.
This explains how the classical $h$-genus obstruction fits into the present
threshold formulation.
\end{remark}

\section{Main Results}\label{sec:main_results}
In this section we state the main results in the threshold form suggested by
the zero-crossing examples.  The surface crossing number does not measure all
of the complexity of a knot when the knot is already carried by the surface.
The invariant starts detecting tunnel complexity only after this complexity
exceeds the amount that can be absorbed by the complementary pieces of the
surface.

\subsection{The zero-crossing stratum}
We first record the obstruction to a crossing-free diagram.

\begin{proposition}[Zero-crossing obstruction]\label{prop:zero}
Let $F\subset S^{3}$ be a closed connected surface, and let
$K\subset S^{3}$ be a knot.  If
\[
c(K;F)=0,
\]
then
\[
t(K)\le \delta(F).
\]
\end{proposition}

\begin{remark}
This proposition is sharp in the simplest examples.  If $K$ is the unknot and
$F=S^2$, then $c(K;F)=0$, $t(K)=0$, and $\delta(F)=0$.  If $K$ is a torus knot
lying on the standard Heegaard torus $F$, then $c(K;F)=0$, $t(K)=1$, and
$\delta(F)=1$.  Thus crossing-free diagrams on a surface should be understood
as diagrams whose knot complexity is carried by the surface rather than by
self-crossings of the projection.
\end{remark}

\subsection{The threshold inequality}
The main theorem says that once the tunnel number exceeds the Heegaard
threshold $\delta(F)$, crossings are forced and are quantitatively bounded
from below.

\begin{theorem}[Threshold inequality]\label{thm:fundamental}
Let $F\subset S^{3}$ be a closed connected surface, and let
$K\subset S^{3}$ be a knot.  If $t(K)>\delta(F)$, then $c(K;F)>0$ and
\[
c(K;F) \ge 2\bigl(t(K)-\delta(F)\bigr)+1.
\]
\end{theorem}

\begin{remark}
When $F$ is a Heegaard surface of $S^{3}$, one has $\delta(F)=g(F)$ by
Remark~\ref{rem:deficiency}.  Hence Theorem~\ref{thm:fundamental} says that
if $t(K)>g(F)$, then
\[
c(K;F)\ge 2(t(K)-g(F))+1.
\]
In particular, a fixed Heegaard surface of genus $g$ can carry crossing-free
representatives only for knots with tunnel number at most $g$.
\end{remark}

\subsection{Existence of knots with arbitrarily large surface crossing number}
As an immediate consequence, the surface crossing number is unbounded for any
fixed closed surface.

\begin{theorem}\label{thm:unbounded}
Let $F\subset S^{3}$ be any closed connected surface, and let $n$ be a positive
integer. Then there exists a knot $K\subset S^{3}$ such that
\[
c(K;F)\ge n.
\]
\end{theorem}

\begin{proof}
Choose $K$ with
\[
t(K)>\delta(F)
\quad\text{and}\quad
2\bigl(t(K)-\delta(F)\bigr)+1\ge n.
\]
Then Theorem~\ref{thm:fundamental} applies and gives $c(K;F)\ge n$.
\end{proof}

\subsection{Explicit realization by connected sums}\label{subsec:explicit}
Let $K_{0}$ be a nontrivial knot with $t(K_{0})=1$, for instance a trefoil
knot, and define
\[
K_m=\underbrace{K_0\#\cdots\#K_0}_{m\text{ times}}.
\]
By results of Morimoto \cite{M} and Scharlemann--Schultens \cite{SS}, the
tunnel number satisfies $t(K_m)\ge m$.  Hence, for each fixed $F$, the
threshold $\delta(F)$ is eventually exceeded.

\begin{proposition}\label{prop:linear_lower_intro}
For every closed connected surface $F\subset S^{3}$ and every integer
$m>\delta(F)$,
\[
c(K_m;F) \ge 2\bigl(m-\delta(F)\bigr)+1.
\]
In particular, $c(K_m;F)\to\infty$ linearly as $m\to\infty$.
\end{proposition}

\begin{proof}
Since $t(K_m)\ge m>\delta(F)$, Theorem~\ref{thm:fundamental} gives
\[
c(K_m;F)\ge 2\bigl(t(K_m)-\delta(F)\bigr)+1
          \ge 2\bigl(m-\delta(F)\bigr)+1.
\]
\end{proof}

\begin{corollary}
For every closed connected surface $F\subset S^{3}$,
\[
\sup_K c(K;F)=\infty.
\]
\end{corollary}

\section{Proof of the Threshold Inequality}
In this section we prove Proposition~\ref{prop:zero} and
Theorem~\ref{thm:fundamental}.  The distinction between $c(K;F)=0$ and
$c(K;F)>0$ arises in the first diagrammatic estimate.  If a diagram has at
least one crossing, one crossing can be chosen as a base crossing and is not
counted among the ascending crossings.  This gives the estimate
\[
\frac{c(K;F)-1}{2}\ge a(K;F).
\]
If there are no crossings, however, there is no base crossing to choose; the
correct conclusion is instead obtained from the bridge and amalgamation
inequalities.

Throughout the proof, let $F\subset S^{3}$ be a closed connected surface with
decomposition
$S^{3}=M_{1}\cup_{F}M_{2}$.

\subsection{Step 1. Bounding ascending number by crossing number}

\begin{lemma}[Positive crossing estimate]\label{lem:asc_cross}
If $c(K;F)>0$, then
\[
\frac{c(K;F)-1}{2} \ge a(K;F).
\]
\end{lemma}

\begin{proof}
Let $D$ be a regular diagram of $K$ on $F$ with $c(K;F)=c>0$.

Choose a crossing of $D$ and specify a base point
and orientation on $K$ so that traversal from the base point first encounters
this crossing as an over-crossing. Call the resulting based, oriented diagram
$D_1$.

Let $D_2$ be the based, oriented diagram obtained by reversing the orientation
of $K$ and sliding the base point so that the same chosen crossing is first
encountered as an over-crossing.

Since the chosen crossing is processed as an over-crossing first in both
$D_1$ and $D_2$, the remaining $c-1$ crossings contribute to the ascending
points. For these crossings, a first-passage under-crossing in $D_1$ becomes
a first-passage over-crossing in $D_2$, and vice versa. Therefore
\[
a(D_1) + a(D_2) = c - 1.
\]

It follows that
\[
\frac{c(K;F)-1}{2}
 = \frac{c-1}{2}
 = \frac{a(D_1) + a(D_2)}{2}
 \ge \min\{a(D_1), a(D_2)\}
 \ge a(K;F).
\]
\end{proof}

\subsection{Step 2. Bounding bridge number by ascending number}

\begin{lemma}[{\cite[Theorem 2.2]{O}}]\label{lem:asc_bridge}
For any knot $K$ and closed surface $F$,
\[
a(K;F) \ge b(K;F) - 1.
\]
\end{lemma}

\begin{proof}
Let $D$ be a regular diagram of $K$ on $F$ with $a(K;F)=a$.  Fix an
orientation and a base point $x$ away from the crossings, and let $d(D)$ be the
descending diagram obtained by changing precisely the $a$ ascending crossings.
Thus, when the diagram is traversed from $x$, the first visit to every crossing
is along the over-strand.

We now realize $d(D)$ in the product neighborhood $N(F)=F\times[-1,1]$.
Cut the diagram at $x$ and parametrize the resulting immersed arc by
$t\in[0,1]$, in the order in which it is traversed from the base point.  Give
this arc the height $1-2t$.  If a crossing is met at parameters $t_1<t_2$, then
the descending condition says that the branch at $t_1$ must be the over-branch;
indeed $1-2t_1>1-2t_2$.  After the usual small local perturbations near the
crossings, this gives an embedded arc in $F\times[-1,1]$ whose projection is
$d(D)$ and whose height is strictly decreasing away from small crossing
neighborhoods.  Finally, close the arc at the base point by a short vertical
arc over $x$.  This vertical arc is contained in a small product fiber
$\{x\}\times[-1,1]$ and is disjoint from the rest of the lifted diagram.

The resulting knot has exactly one local maximum and one local minimum with
respect to the height coordinate.  Moreover, the upper and lower pieces are
boundary-parallel in the two product regions.  For example, the upper piece is
obtained by joining an initial segment of the decreasing lifted arc to a
subarc of the vertical closing arc; projecting this piece to $F\times\{0\}$
and taking the product trace gives a bridge disk.  If the projected path has
self-intersections on $F$, the different sheets occur at different heights,
and a sufficiently thin product trace is embedded.  The same argument applies
to the lower piece.  Hence $d(D)$ gives a $1$-bridge presentation with respect
to $F$.  No assertion is made that this knot is trivial in $S^3$; on a
higher-genus surface a descending diagram may represent an essential homology
class of $F$.

To recover the original knot $K$, restore the original crossing information at
the $a$ ascending crossings.  This is done inside pairwise disjoint small
3-balls in $N(F)$ by pushing the appropriate strand upward over the other
strand.  Each such local move introduces at most one additional local maximum
and preserves the boundary-parallel bridge disks outside the small ball, with
the obvious local disk added in the ball.  Therefore $K$ has a bridge
presentation with at most $1+a$ bridges.  Thus
\[
b(K;F)\le 1+a,
\]
and hence $a(K;F)\ge b(K;F)-1$.
\end{proof}

\subsection{Step 3. Amalgamation of Heegaard splittings}

\begin{lemma}\label{lem:amalgamation}
For any knot $K$ and closed surface $F$,
\[
t(K) \le \delta(F) + b(K;F) - 1.
\]
\end{lemma}

\begin{proof}
Write
\[
S^3=M_1\cup_F M_2,
\]
and let $g_i=g(M_i;F)$ for $i=1,2$.
Choose a minimal relative Heegaard splitting
\[
M_i=C_i\cup_{\Sigma_i} H_i
\]
such that $C_i$ is a compression body with $\partial_-C_i=F$ and $H_i$ is a handlebody.  Then $g(\Sigma_i)=g_i$.

By the standard amalgamation construction for Heegaard splittings of manifolds with boundary, these splittings determine a Heegaard splitting of
$M_1\cup_F M_2=S^3$
with Heegaard surface $F'$.  We use the usual genus formula for amalgamation; see Schultens \cite{ScAdd} and Yang--Lei \cite[Introduction]{YL}.
Thus
\[
g(F')=g(\Sigma_1)+g(\Sigma_2)-g(F)
      =g(M_1;F)+g(M_2;F)-g(F)
      =\delta(F).
\]

Let $b=b(K;F)$, and choose a $b$-bridge position of $K$ with respect
to $F$ in the sense of Definition~\ref{def:surface_bridge}.  Thus the upper
bridge arcs are boundary-parallel in $F\times[0,1]$ and the lower bridge arcs
are boundary-parallel in $F\times[-1,0]$.  Fix bridge disks
\[
\Delta_1^+,
\ldots,
\Delta_b^+,
\Delta_1^-,
\ldots,
\Delta_b^-
\]
for these arcs.  Each $\Delta_j^+$ has boundary equal to an upper bridge arc
together with an arc on $F$, and similarly for $\Delta_j^-$.

Geometrically, $F'$ is obtained from $F$ by tubing along the $1$-handles
arising from the compression bodies $C_1$ and $C_2$; see
Figure~\ref{fig:amalgamation}.  We choose the attaching disks of these tubes
in $F$ after the bridge position has been fixed.  Since the projected bridge
diagram and the boundary arcs of the bridge disks form a finite graph on
$F$, the attaching disks may be taken to be small disks in the complement of
this graph.  In particular, the tube feet are disjoint from $K\cap F$ and from
the arcs $\Delta_j^\pm\cap F$.

It remains to ensure that the interiors of the bridge disks are also disjoint
from the tube neighborhoods.  Put the bridge disks and the tube neighborhoods
in general position.  Their intersections are compact 1-manifolds.  Closed
intersection curves are removed by an innermost-circle argument on the bridge
disks, using the fact that the meridian disks of the tube feet are disks.
After all circle components have been removed, any remaining intersection arc
is treated by an outermost-arc argument.  An outermost such arc cuts off a disk
from some bridge disk whose other boundary arc lies either on a tube foot or on
the side of a tube neighborhood; pushing this subdisk across the corresponding
simply connected disk removes the intersection.  The isotopy is supported in a
small product neighborhood of the tube and is taken rel boundary, so it does
not move the bridge arc of $K$ or the boundary arc lying on $F$.  Repeating
this process makes all bridge disks disjoint from the tube neighborhoods.

After the tubes have been attached, the same upper and lower bridge arcs remain
boundary-parallel to the new surface $F'$.  Indeed, the bridge disks just
constructed are disjoint from the tubes, and their boundary arcs on $F$ lie in
$F\setminus\{\text{tube feet}\}$, which is naturally contained in $F'$.  Hence
the same disks, viewed in the complementary compression bodies determined by
$F'$, are bridge disks for the same bridge arcs.  Therefore the chosen
$b$-bridge presentation with respect to $F$ induces a $b$-bridge presentation
with respect to the Heegaard surface $F'$.  Consequently
\[
b(K;F')\le b(K;F)=b.
\]

Now $F'$ is a Heegaard surface of genus $\delta(F)$.
Therefore a $b(K;F')$-bridge position of $K$ with respect to $F'$ gives a
\[
(\delta(F),\,b(K;F'))
\]
-decomposition of $K$ in the sense of Morimoto--Sakuma--Yokota \cite{MSY}.
Applying the standard inequality for a $(g,b)$-decomposition, we obtain
\[
t(K)\le \delta(F)+b(K;F')-1
      \le \delta(F)+b(K;F)-1.
\]
This proves the lemma.
\end{proof}
\begin{figure}[htbp]
    \centering
    \begin{tikzpicture}[scale=1.2, thick]
        
        \fill[blue!5] (-4, 0) rectangle (4, 3);
        \fill[green!5] (-4, 0) rectangle (4, -3);

        \draw[very thick] (-4,0) -- (4,0) node[right] {$F$};

        \draw[fill=white] (-2.5,0) .. controls (-2.5,1.5) and (-1.5,1.5) .. (-1.5,0);
        \draw[fill=white] (1.5,0) .. controls (1.5,1.5) and (2.5,1.5) .. (2.5,0);
        
        \draw[fill=white] (-0.5,0) .. controls (-0.5,-1.5) and (0.5,-1.5) .. (0.5,0);

        \draw[ultra thick, blue, dashed] (-4,0.1) 
            -- (-2.6, 0.1) .. controls (-2.6,1.6) and (-1.4,1.6) .. (-1.4, 0.1)
            -- (-0.6, 0.1) .. controls (-0.6,-1.4) and (0.6,-1.4) .. (0.6, 0.1)
            -- (1.4, 0.1) .. controls (1.4,1.6) and (2.6,1.6) .. (2.6, 0.1)
            -- (4,0.1) node[above right] {$F'$};

        \draw[ultra thick, red] (-3.2, -1) .. controls (-3.5, 0) and (-2.8, 0) .. (-3.2, 1) node[above] {$K$};
        \draw[ultra thick, red] (-1, -1) .. controls (-1.2, 0) and (-0.8, 0) .. (-1, 1);
        \draw[ultra thick, red] (1, -1) .. controls (0.8, 0) and (1.2, 0) .. (1, 1);
        \draw[ultra thick, red] (3.2, -1) .. controls (2.8, 0) and (3.5, 0) .. (3.2, 1);

        \node at (0, 2) {$H_1$};
        \node at (0, 0.5) {$C_1$};
        \node at (-3, -2) {$H_2$};
        \node at (-3, -0.5) {$C_2$};
        \node at (0, -2) {$1$-handle of $C_2$};
        \node at (2, 2) {$1$-handle of $C_1$};

    \end{tikzpicture}
\caption{Schematic illustration of the amalgamation of Heegaard splittings of
$M_1$ and $M_2$ along $F$. The resulting Heegaard surface $F'$ is obtained by
tubing $F$ according to the $1$-handles coming from the compression bodies
$C_1$ and $C_2$.  The red arcs indicate bridge arcs schematically; the figure
is not intended to depict the full closed knot.}
    \label{fig:amalgamation}
\end{figure}

\subsection{Step 4. Completion of the proof}

\begin{proof}[Proof of Proposition~\ref{prop:zero}]
Assume that $c(K;F)=0$.  Then $K$ admits a diagram on $F$ with no crossings.
For this diagram there are no ascending crossings, and hence $a(K;F)=0$.  By
Remark~\ref{rem:bridge_positive}, $b(K;F)\ge 1$; in fact the next inequality
will force $b(K;F)=1$ in this case.  Lemma~\ref{lem:asc_bridge} and
Lemma~\ref{lem:amalgamation} give
\[
0=a(K;F)\ge b(K;F)-1\ge t(K)-\delta(F).
\]
Therefore $t(K)\le\delta(F)$.
\end{proof}

\begin{proof}[Proof of Theorem~\ref{thm:fundamental}]
Assume that $t(K)>\delta(F)$.  Proposition~\ref{prop:zero} first implies that
$c(K;F)\ne 0$.  Thus $c(K;F)>0$, and Lemma~\ref{lem:asc_cross} applies.
Combining Lemmas~\ref{lem:asc_cross}, \ref{lem:asc_bridge}, and
\ref{lem:amalgamation}, we obtain
\[
\frac{c(K;F)-1}{2}\ge a(K;F)\ge b(K;F)-1\ge t(K)-\delta(F).
\]
Multiplying by $2$ and adding $1$ gives
\[
c(K;F)\ge 2\bigl(t(K)-\delta(F)\bigr)+1.
\]
This completes the proof.
\end{proof}

\section{Examples and Sharpness of the Estimate}\label{sec:sharpness}
In this section we study the family $K_m$ introduced in Section~\ref{subsec:explicit}.

\subsection{Linear lower bound}\label{subsec:family}
By Lemma~\ref{lem:tunnel_bound} below, due to Morimoto \cite{M} and
Scharlemann--Schultens \cite{SS}, we have $t(K_m) \ge m$.

\begin{lemma}[{\cite{M, SS}}]\label{lem:tunnel_bound}
$t(K_{m}) \ge m$.
\end{lemma}

If $m>\delta(F)$, then $t(K_m)\ge m>\delta(F)$, so the zero-crossing case
of Theorem~\ref{thm:fundamental} is impossible.  Hence
\[
c(K_m;F) \ge 2(m-\delta(F))+1,
\]
so $c(K_m;F)\to\infty$ linearly in $m$.

\subsection{Linear upper bound}

\begin{proposition}\label{prop:upper_bound}
For every closed surface $F\subset S^{3}$, there exists a constant $C > 0$ (independent of $m$) such that $c(K_{m};F) \le C \cdot m$.
\end{proposition}

\begin{proof}
By Remark~\ref{rem:finite_crossing}, $c(K;F)\le c(K)$ for any knot $K$.  We
only use the standard subadditivity inequality for the classical crossing
number: juxtaposing diagrams gives
\[
c(K_1\#K_2)\le c(K_1)+c(K_2).
\]
The corresponding equality is the well-known open additivity problem for
crossing number.  Iterating the inequality gives
\[
c(K_m)\le m\,c(K_0).
\]
Setting $C=c(K_0)$ gives $c(K_m;F)\le C m$.
\end{proof}

\begin{theorem}\label{thm:sharpness}
For the family $K_{m}$, there exist positive constants $C_1$ and $C_2$,
and an integer $m_0$, depending on $F$ but independent of $m$, such that
\[
C_1 \cdot m \le c(K_{m};F) \le C_2 \cdot m
\]
for all $m\ge m_0$.
\end{theorem}

\begin{proof}
For $m>\delta(F)$, Theorem~\ref{thm:fundamental} and Lemma~\ref{lem:tunnel_bound} give
\[
c(K_m;F)\ge 2(m-\delta(F))+1.
\]
The inequality $2(m-\delta(F))+1\ge m$ is equivalent to
$m\ge 2\delta(F)-1$.  At the same time, the threshold theorem requires
$m>\delta(F)$, i.e. $m\ge \delta(F)+1$ since $m$ is an integer.  Hence we may
take
\[
m_0=\max\{\delta(F)+1,\,2\delta(F)-1\}
\]
and $C_1=1$ for the lower bound.  The upper bound is
Proposition~\ref{prop:upper_bound}; for instance one may take $C_2=c(K_0)$.
\end{proof}

\section{Discussion and Further Questions}

\subsection{Hierarchy of surface invariants}
The proof gives the following hierarchy.  Once $t(K)>\delta(F)$, the
zero-crossing stratum is excluded and one has
\[
\frac{c(K;F)-1}{2}\ge a(K;F)\ge b(K;F)-1\ge t(K)-\delta(F).
\]
Thus, above the threshold $\delta(F)$, the diagrammatic invariants
$c(K;F)$ and $a(K;F)$ and the bridge invariant $b(K;F)$ all detect tunnel
complexity.  At or below the threshold, however, the knot may be carried by
$F$ without crossings, and the crossing number alone no longer records this
complexity.

\subsection{Behavior under connected sums}

The additivity of the classical crossing number,
\[
c(K_1 \# K_2) = c(K_1) + c(K_2),
\]
is still open. The behavior of the surface invariants $c(K;F)$, $a(K;F)$, and $b(K;F)$ under connected sums is also unclear for a general closed surface $F$.

When $F$ is a Heegaard surface of $S^{3}$, Doll \cite{Doll} proved an additivity formula for the surface bridge number $b(K;F)$. For general closed surfaces, and especially for $c(K;F)$ and $a(K;F)$, much less seems to be known.

\begin{question}
How do the surface crossing number $c(K;F)$ and the surface ascending number $a(K;F)$ behave under the connected sum $K_1 \# K_2$? Furthermore, under what general conditions does $b(K;F)$ exhibit additivity when $F$ is an arbitrary closed surface rather than a Heegaard surface of $S^{3}$?
\end{question}


\subsection{Comparison with the trunk of a knot}\label{subsec:trunk}
The trunk of a knot $K$, introduced in \cite{O_trunk}, is a minimax invariant defined using intersections with level spheres. One can define an analogous quantity for a closed surface $F$.

\begin{definition}
Let $N(F) \cong F \times [-1, 1]$ be a regular neighborhood of $F$ in $S^{3}$, and let $h \colon N(F) \to [-1, 1]$ be the projection onto the second factor. We isotope $K$ into $N(F)$ such that the restriction $h|_{K}$ is a Morse function. The \emph{surface trunk} $\operatorname{trunk}(K;F)$ is defined as the minimal value of
\[
\max_{t \in [-1,1]} |K \cap (F \times \{t\})|,
\]
where the minimum is taken over all such isotopic positions of $K$ in $N(F)$.
\end{definition}

The surface bridge number bounds the surface trunk from above:
\[
\operatorname{trunk}(K;F) \le 2b(K;F).
\]
This suggests that trunk-type invariants behave differently from the surface crossing number considered here.  In the classical case, Davies and Zupan \cite{DZ} proved
\[
\operatorname{trunk}(K_1 \# K_2) = \max\{\operatorname{trunk}(K_1), \operatorname{trunk}(K_2)\}.
\]
Thus classical trunk stays bounded under iterated connected sums of a fixed knot, whereas the tunnel number need not.  This contrast explains why the present lower bound is formulated in terms of crossing, ascending, and bridge data rather than trunk alone.

\subsection{Generalization to surfaces with boundary}
A similar argument should apply to compact orientable surfaces with boundary
properly embedded in $S^{3}$.  Replacing the two complementary pieces of a
closed surface by the exterior
$E(F)=S^{3}\setminus \operatorname{int}N(F)$ suggests a threshold statement of
the following form: if $t(K)>g(E(F))$, then
\[
c(K;F)\ge 2\bigl(t(K)-g(E(F))\bigr)+1,
\]
where $g(E(F))$ is the Heegaard genus of the exterior of the surface.  The
corresponding zero-crossing statement would say that a crossing-free diagram
forces $t(K)\le g(E(F))$.  We do not pursue this here.

\subsection{Generalization to spatial graphs}
One can also ask for an analogue for spatial graphs.  A natural tunnel number
for a spatial graph $G$ may be defined by taking the minimum number of arcs in
the exterior of a regular neighborhood $N(G)$ whose addition makes the exterior
a handlebody; equivalently, one may formulate it in terms of the Heegaard genus
of $S^3\setminus \operatorname{int}N(G)$.  Since vertices change the local
bridge structure, any graph version should include a correction term depending
on the combinatorics of $G$, for instance on the number of vertices and their
valences.  Thus one expects a lower bound of the form
\[
c(G;F) \ge 2\bigl(t(G)-\delta(F)\bigr) + C_G,
\]
where $C_G$ is a graph-theoretic correction term.  Determining the correct
normalization of $t(G)$ and $C_G$ is part of the problem.

\subsection{Necessity of the manifold condition}
The assumption that $F$ is a genuine surface is essential. If singular spaces are allowed, the conclusion fails.

Dynnikov \cite{D} showed that every knot embeds without crossings in a 3-page book, namely a 2-complex formed by three half-planes meeting along a common axis. It is also known that there exists a branched surface $B \subset S^{3}$ that contains all knots as embedded curves \cite{G}. In both settings the crossing number is zero for every knot.

\subsection{Asymptotic behavior for prime knots}
The family $K_m$ suggests a linear relation between surface crossing number and tunnel number for some families. This cannot hold for all knots: the planar crossing number of a torus knot $T(p, q)$ or a 2-bridge knot can be arbitrarily large even though the tunnel number is always one.

\begin{question}
For which classes of prime knots does $c(K;F) = \Theta(t(K))$ hold? In particular, what geometric or topological conditions on a family of prime knots ensure that $c(K;F)$ is bounded above by a linear function of $t(K)$?
\end{question}

\subsection{Minimality of alternating projections}

The Kauffman--Murasugi--Thistlethwaite theorem says that a reduced alternating diagram on $S^2$ realizes the minimal crossing number. For a closed surface $F \subset S^3$, one should at least require that the diagram be cellular, namely that each component of $F \setminus D$ is a disk \cite{Oz06}. This leads to the following question.

\begin{question}
Let $F \subset S^3$ be a closed surface. If a knot $K$ admits a reduced alternating regular projection $D$ on $F$ such that each region of $F \setminus D$ is a disk, does $D$ realize the surface crossing number, namely $c(D) = c(K;F)$?
\end{question}

\section*{Acknowledgements}
The author thanks the editors and referees whose comments helped improve the
formulation and exposition of this work.  The author used ChatGPT (OpenAI) for
language editing, organizational suggestions, and preliminary consistency
checks during the preparation of this manuscript.  All mathematical content,
including the definitions, statements, proofs, and references, was reviewed and
verified by the author, who takes full responsibility for the paper.

\end{document}